\documentclass[12pt]{article}
\usepackage[utf8]{inputenc}
\usepackage{amsmath,amssymb,amsthm}
\usepackage{hyperref}
\usepackage{graphicx}
\renewcommand\le{\leqslant}
\renewcommand\ge{\geqslant}
\newcommand\eps{\varepsilon}
\newcommand\R{\mathbb R}
\newcommand\wphi{\widetilde{\varphi}}

\newcommand{\embed}[2]{\stackrel{#1,\,#2}{\lhook\joinrel\longrightarrow}}
\newcommand{\embednet}[3]{\stackrel{#1,\,#2,\,#3}{\lhook\joinrel\longrightarrow}}
\newtheorem*{theorem}{Theorem}
\newtheorem{proposition}{Proposition}
\newtheorem{lemma}{Lemma}
\newtheorem*{corollary}{Corollary}

\theoremstyle{definition}
\newtheorem*{procedure}{Procedure}
\newtheorem*{definition}{Definition}

\title{Recovery of regular ridge functions on the ball}
\author{Tatyana~Zaitseva\thanks{Moscow Center for Fundamental and Applied
Mathematics, Laboratory ``High-dimensional Approximation and Applications'' of
Lomonosov Moscow State University. Email: zaitsevatanja@gmail.com},
Yuri Malykhin\thanks{Steklov Mathematical Institute, Laboratory
``High-dimensional Approximation and Applications'' of
Lomonosov Moscow State University. Corresponding author. Email: malykhin@mi-ras.ru}, Konstantin~Ryutin\thanks{Moscow Center for Fundamental and Applied
Mathematics,  Laboratory ``High-dimensional Approximation and Applications'' of
Lomonosov Moscow State University. Email: kriutin@yahoo.com}}

\begin{document}
\maketitle
\begin{abstract}
    We consider the problem of the uniform (in $L_\infty$) recovery of ridge
    functions $f(x)=\varphi(\langle a,x\rangle)$, $x\in B_2^n$, using noisy
    evaluations $y_1\approx f(x^1),\ldots,y_N\approx f(x^N)$. It is known that
    for classes of functions $\varphi$ of finite smoothness the
    problem suffers from the curse of dimensionality: in order to provide good
    accuracy for the recovery it is necessary to make exponential number of
    evaluations.  We prove that if $\varphi$ is analytic in a neighborhood of
    $[-1,1]$ and the noise is very small, $\eps\le\exp(-c\log^2n)$, then there is an efficient algorithm that
    recovers $f$ with good accuracy using $O(n\log^2n)$ function
    evaluations.
\end{abstract}



\section{Introduction}

\paragraph{Ridge functions recovery.}

We consider ridge functions (plane waves) on the ball
$B_2^n=\{x\in\mathbb{R}^n\colon|x|\le 1\}$, i.e. functions of the form
$f(x)=\varphi(\langle a,x\rangle)$, where $a\in\mathbb R^n$, $|a|=1$ is a fixed
vector and $\varphi$ is a function on $[-1,1]$ (as usual, $\langle x,y\rangle$
is the scalar product on $\mathbb{R}^n$). The problem is to recover $f$ from its
$N$ noisy values at some points: $y_1\approx f(x^1),\ldots,y_N\approx f(x^N)$. The recovery algorithm is allowed
to choose points $x^1,\ldots,x^N$ for function evaluation; the algorithm is called
\textit{adaptive} if the point $x^{i+1}$ may depend on the previously obtained values
$y_1,\ldots,y_i$. When we say that our algorithm makes an evaluation of
the unknown function we  mean that it prescribes the point and receives the
approximate value of the function at this point.

The difficulty is that when $\log N=o(n)$, for any points
$x^1,\ldots,x^N\in B_2^n$ there exists a unit vector $a$ such that
$\max_j|\langle a,x^j\rangle| = o(1)$, as $n\to\infty$. Therefore we can measure
$\varphi$ only in a small neighborhood of zero and we cannot distinguish a
nontrivial function supported outside this neighborhood from the identically zero function.
Therefore, if one knows only that  $\varphi$ is smooth, then any recovery method
requires exponential (with respect to the dimension) number of evaluations.

In the theory of function recovery the number of function evaluations sufficient
to recover the unknown function with a given accuracy is studied.  If
its dependence  on the dimension is exponential, then it is called the  curse of dimensionality. Results on the existence of recovery
algorithms with a small number of evaluations (say, polynomially depending on the
dimension of the space) or, conversely, that any algorithm requires a large
(exponential) number of evaluations are actively studied within the Information-based
Complexity theory.
See~\cite{NW16} for some statements and results in this direction.

For $r>0$ let $\mathrm{Lip}(r)$ be the class of functions  $g\colon[-1,1]\to
\mathbb{R}$, with $m=\lceil r-1\rceil$ continuous derivatives on $[-1,1]$,
s.t. $\|g^{(j)}\|_\infty\le 1$, $j=0,\ldots,m$, and $g^{(m)}$
is $(r-m)$--H\"{o}lder with constant $1$.

The recovery in $L_\infty$ for smooth ridge functions on the
cube,
\begin{equation}\label{ridge_cube}
    f(x)=\varphi(\langle a,x\rangle),\quad x\in [-1,1]^n,\;\|a\|_1 = 1,\;
\varphi\colon [-1,1]\to\mathbb R,
\end{equation}
was studied by A.~Cohen, I.~Daubechies, R.~DeVore, G.~Kerkyacharian, D.~Picard in~\cite{CDD12}.
The authors assume that $\varphi\in \mathrm{Lip}(r)$ and the vector $a$ is
nonnegative: $a_i\ge 0$, $i=1,\ldots,n$. This condition is very restrictive and it permits to
 recover  $\varphi(t)$ at any point  $t\in[-1,1]$ as: $\varphi(t) = f(x_t)$,
$x_t := (t,t,\ldots,t)$. The most interesting results in~\cite{CDD12} were obtained
under the additional restriction that $a$ is bounded in the weak-$\ell_p^n$ norm, $p\in(0,1)$.

These studies were continued by B.~Doerr and S.~Mayer in~\cite{DM21}. They dropped
the non-negativity condition and replaced weak-$\ell_p$ condition by the following:
\begin{equation}\label{uslov_a}
    \|a\|_p\le M\quad\mbox{for some $p\in(0,1)$ and $M\ge 1$.}
\end{equation}
Let $\mathcal R(r,p,M)$ be the class of ridge functions  $f$ on the
cube,~(\ref{ridge_cube}), with $\varphi\in\mathrm{Lip}(r)$ and vector $a$
satisfying ~(\ref{uslov_a}). The algorithm
from~\cite{DM21} gives an approximation  $\widetilde{f}$, such that
$$
\sup_{f\in\mathcal R(r,p,M)} (\mathsf{E}\|f-\widetilde{f}\|_\infty^2)^{1/2}
\le C(r,p,M)(\log N)^{-r(1/p-1)}\quad\mbox{for $N\ge n$.}
$$
On the other hand it was proven that without~(\ref{uslov_a}) we have the curse
of dimensionality: a similar recovery error for the class of ridge
functions of the form~(\ref{ridge_cube}) with $\varphi\in\mathrm{Lip}(r)$ is at
least $\eps_0>0$, when  $N\le C\exp(n/8)$.

Let us consider ridge functions on the ball:
$$
f(x)=\varphi(\langle a,x\rangle),\quad |a|=1,\;x\in B_2^n.
$$

The corresponding recovery problem, even in a more general setting, was
considered by M.~Fornasier, K.~Schnass and J.~Vybiral in~\cite{FSV12}.
They aim to recover functions of the form $f(x)=\varphi(Ax)$, where $A$ is some
fixed $k\times n$ matrix and $\varphi$ is a function of  $k$ variables. Let us
restrict ourselves to the case $k=1$ (ridge functions).
The authors of~\cite{FSV12} introduce the following quantity 
\begin{equation}\label{alpha}
\alpha := \int_{S^{n-1}}|\varphi'(\langle a,x\rangle)|^2\,d\mu_{S^{n-1}}(x)
\end{equation}
(it does not depend on $a$).
They proved that if $a$ satisfies~(\ref{uslov_a}) (the case of $p=1$ is possible),
$\varphi\in C^2$ and $\alpha$ is not close to zero (e.g. $|\varphi'(0)|$ is
separated from zero) then the efficient recovery of $f$ is possible. Roughly speaking,  $N\asymp
(\omega\alpha)^{-3}$ evaluations are sufficient in order to achieve the error
$\le\omega$ with high probability (when $p=1$). Later, H.~Tyagi and
V.~Cevher in~\cite{TC14} managed to get rid of the condition~(\ref{uslov_a}).

We remark that the functions considered in our paper can be very small in the
neighborhood of zero and the quantity $\alpha$ can also be very small. Given the
admissible error $\omega$, the class of analytic functions contains a function
with oscillation more than $\omega$, but $\alpha \le \exp(-c\log(1/\omega)\log n)$. 
This explains why we require the noise level to be subexponentially small
(it is a drawback for practical applications).

In~\cite{MUV15} different approximation characteristics for several classes of
ridge functions on the ball were estimated. As a corollary the formal proof of the
curse of dimensionality for the corresponding recovery problem was given.

A closely related problem in statistics is known as a ``single index model regression''. 
We briefly describe its formulation. Suppose $X$ is a random vector in $\R^n$ and
$Y$ is a random variable. It is assumed that the regression function
$f(x) = \mathsf{E}(Y|X=x)$ has the form $f(x) = g(\langle a,x\rangle)$, i.e.,
it is a ridge function; in the language of statistics, such an assumption is called
the single index model. It is required to construct a regression, i.e.,
approximate $f$ from a sample $\{(X_i, Y_i)\}_{i=1}^N$.
We note the following important distinctions between the regression
problem and our
problem: (1) we are not free to choose $X_i$; these points are a random sample from the
distribution $\mu_X$ unknown to us;
(2) the error is measured in the $L_2(\mu_X)$-norm rather than in the uniform norm.
See~\cite{GL07} for the statistical setting.

Ridge functions appear also in the study of neural networks; this connection (in
the context of the approximation theory) is discussed in the paper \cite{P99}.
Indeed, the output of a one-neuron neural network with the activation function
$\sigma$ is the ridge function: $x\mapsto \sigma(\langle w,x\rangle)$. Let us
remark  that the problem of almost-optimal fitting of neuron weights vector $w$
to a training set $\{(x_i,y_i)\}_{i=1}^N$ is NP-hard, see \cite{Sima} (for
$\ell_2$-norm). However, this problem differs from our setting in many aspects.

We remark that the well-known \textit{Phase Retrieval Problem} (see \cite{CW}) for real vectors  can be
reformulated as a particular case $\varphi(u)=u^2$ of the ridge function recovery
problem. It is not easy to compare our results with the state of the art in the
Phase Retrieval Problem since the settings are different (e.g. in the latter
problem the focus is on non-adaptive measurements).

 Many fundamental properties of the classes of ridge functions were studied   by
 A.Pinkus and his co-authors. The recovery problem was considered in the paper
 \cite{BP}; see also the book \cite{Pbook}. We also recommend the survey
 \cite{KKM}: it contains interesting results about the representation
 by sums of ridge functions and their approximation properties.

\paragraph{Regular ridge functions.}
S.V. Konyagin (personal communication) suggested to consider the case of
regular ridge functions on the ball, i.e. $f(x)=\varphi(\langle a,x\rangle)$ with
regular (analytic) $\varphi$. 

The recovery for classes of analytic functions was actively studied since the 1960s.
The main focus was on the best possible method. The details about the results
can be found in \cite{Os00}. 
In the Information-based Complexity field there are different results on tractability of integration and approximation problems for certain classes of analytic functions. 
We should mention the papers by J. Dick, P. Kritzer, F. Pillichshammer, H. Wo\'{z}niakowski \cite{DKPW}  and Vybiral \cite{Vyb14}.

In our paper we apply the recovery method and
error estimates from \cite{DT19}: an analytic function is extrapolated
from its noisy evaluations on the grid using least-squares fitted
algebraic polynomials.


Let $H(\sigma,Q)$ be the class of analytic functions  $\varphi$ in
$\Pi_\sigma=\{z=t+iy\colon |t|\le 1+\sigma, |y|\le \sigma\}$, such that
$|\varphi(z)|\le Q$ in $\Pi_\sigma$ and $\varphi(z)$ is real--valued for real $z$.
For this class we do have a polynomial recovery algorithm provided  the evaluation errors are
sufficiently small.

Given a class $\Phi$ of functions $\varphi\colon[-1,1]\to\R$, we denote by
$\mathcal R(\Phi,B_2^n)$ the class of ridge functions $f\colon B_2^n\to\R$ of the form
$f(x)=\varphi(\langle a,x\rangle)$, where $a\in\R^n$, $|a|=1$ and
$\varphi\in\Phi$.

\begin{theorem}
    Let $n\in\mathbb N$, $\sigma\in (0,1]$, $Q\ge 1$, $\delta_*\in(0,\frac12)$, $\omega_*\in(0,\frac12)$.
    There is a probabilistic adaptive algorithm: for any function $f\in\mathcal
    R(H(\sigma,Q),B_2^n)$ it uses $N$
    evaluations of $f$ with errors not exceeding $\eps$, where
    \begin{equation}\label{th_eps}
    \eps := Q\exp(-C(\sigma)\log(Qn/\omega_*)\log n),
    \end{equation}
    \begin{equation}\label{th_N}
    N := \lceil C(\sigma)\log^2(Qn/\omega_*)(\log(1/\delta_*)+n) \rceil
    \end{equation}
    and outputs an approximation  $\widetilde f$, such that with probability $\ge
    1-\delta_*$ we have
    $$
    \max_{x\in B_2^n}|\widetilde{f}(x)-f(x)|\le \omega_*.
    $$
\end{theorem}

We remark that $\widetilde{f}$ is of the form  $\widetilde{f}(x) = \wphi(\langle
\widetilde{a},x\rangle)$, with $\wphi$ an algebraic polynomial of degree not
exceeding $C(\sigma)\log(Q/\omega_*)$.  The number of operations for the
algorithm depends polynomially on $n$, $Q/\omega_*$ and $\log(1/\delta_*)$ for
fixed $\sigma$.

The function  $f$ is invariant under  $(a,\varphi(x))\mapsto
(-a,\varphi(-x))$.
We recover either  $a$ and $\varphi$ or $-a, \varphi(-x)$.

\begin{corollary}
    Let $\varkappa>0$.  There is a polynomial algorithm that uses at most
    $C(\sigma,\varkappa)n\log^2n$ evaluations of an unknown function $f\in
    R(H(\sigma,Q),B_2^n)$ with errors $\le Q\exp(-C(\sigma,\varkappa)\log^2 n)$
    and outputs an approximation $\widetilde f$, such that
    $\|\widetilde{f}-f\|_{L_\infty(B_2^n)}\le Qn^{-\varkappa}$ 
    with probability $\ge 1-2e^{-n}$.
\end{corollary}
   
We did not try to optimize all the steps of the algorithm.  Our main point is
the possibility of the effective (of  polynomial complexity)  recovery   of regular
ridge functions. We also developed a new method involving global
properties of functions (the so-called embedding) and order statistics.
These results were announced in~\cite{ZMR20}.

In section~\ref{sect_practice} we discuss the
numerical implementation of the algorithm.  
We remark that in the program we find embedding coefficients using an effective procedure that
requires minimization of a polynomial, see~\eqref{effective_embed}.

\paragraph{Notations and useful facts.}
By homogeneity we may assume that $Q=1$. Indeed, we will show that the
algorithm works for $Q=1$. For a general $Q\ge 1$ (which is given), we can divide all evaluated values
$y_k$ by $Q$ and apply our algorithm to recover $f/Q\in H(\sigma,1)$ with error $\omega_*/Q$.

Fix $\sigma\in(0,1]$. We may further assume that $n$ is large enough and $\omega_*$ and $\delta_*$ are small enough.
Indeed, suppose that our algorithm works for some $n^\circ$, $\omega_*^\circ$,
$\delta_*^\circ$ and requires $N^\circ$ function evaluations with error
$\eps^\circ$. 
Then the same algorithm works also for any $n\le n^\circ$ (we can treat the
vector $a\in S^{n-1}$ as a vector in $S^{n^\circ-1}$ by adding zero coordinates)
$\omega_*\ge\omega_*^\circ$, $\delta_*^\circ\ge\delta_*$. Our Theorem holds for
$n$, $\omega_*$, $\delta_*$ provided that the number of evaluations given by the
expression~\eqref{th_N} is at least $N^\circ$, and $\eps$ from~\eqref{th_eps}
is at most $\eps^\circ$. This can be achieved if we take $C(\sigma)$ large enough.

For the rest of the paper $\varphi$ denotes some unknown function from
$H(\sigma,1)$, $\sigma\in(0,1]$, and $a$ the unknown vector of coefficients of
the ridge function  $f(x)=\varphi(\langle a,x\rangle)$, $x\in B_2^n$, $|a|=1$.

Different quantities  that are
produced by the algorithm  will be denoted by variables with  ``tildes'': e.g.
$\wphi_i$, $\widetilde\Delta^i_{h,\nu}$ (they depend on the taken  samples of our function $f$).

In the rest of the paper for any vector $\gamma \in B_2^n$ we denote   $ v_\gamma
:= n^{1/2}\langle a,\gamma \rangle,$ and $\wphi_\gamma$ is the approximation  for  the function  $\varphi(v_\gamma t)$ satisfying  \eqref{phi_approx} (see the first step of the algorithm).
        
We will fix some constants $b,B>0$; it is supposed that $b^{-1}$ and $B$ are large
enough. Namely, they satisfy the condition \eqref{bB_def} given below; in fact one
can take  $b=0.01$, $B=5$. We call a real number $v$
\textit{typical} if  $b \le |v|\le B$.

Let $\sigma_1 := \sigma\frac{b}{4B}$.


Throughout the paper  $t$ and $x$ are real variables and  $z$ denotes a complex variable.
Therefore the set  $\{|t|\le h\}$ is a segment, and $\{|z|\le h\}$ is a disk on the complex plane.

We denote by $c,c_1,\ldots,C,C_1,\ldots$ positive reals (their values may
differ from line to line).

The functions from  $H(\sigma,1)$ are $L_\sigma$-Lipschitz on  $[-1,1]$, and 
$L_\sigma\le C\sigma^{-1}$ (the explicit dependence on $\sigma$ is not important for us).


Let $E_\rho$ be the ellipse with focii $\pm 1$ and the sum of semi-axes $\rho$.
We note that the major semi-axis of  $E_\rho$ equals 
$R=\frac12(\rho+\rho^{-1})$; and we have $\rho=R+\sqrt{R^2-1}$. It is clear that $E_{1+\sigma}\subset \Pi_\sigma$. When we say that some function $\psi$ is analytic in $E_\rho$ and $|\psi|\le C$  there, we mean that it is analytic in the open domain bounded by $E_\rho$ and the estimate holds in the closed domain.

We make use of different probabilistic notions and constructions.  For any random variable $\xi$ we denote by  $\mathsf{Law}(\xi)$ its distribution.
 $\Phi$ denotes the distribution function of the standard gaussian random
 variable; by $\Phi^*$ we denote the distribution function  of  $|\xi|$,
 $\xi\sim\mathcal N(0,1)$, i.e. $\Phi^*(x)=2(\Phi(x)-1/2)$, $x\ge 0$.

For describing the algorithm and estimating its accuracy we use (global)
parameters $N_1$, $N_2$, $N_3$, $M$, $M_1$ (sufficiently large natural numbers) and
$\omega_1$, $\omega_2$, $\omega_3$ (sufficiently small positive reals)
that will be chosen in order to satisfy different inequalities required for our
algorithm. All of them may depend on $\sigma$, $n$, $\omega_*$,
with the sole exception of $N_2$, that depends only on $\delta_*$.  The
meaning of these parameters will be clear from the scheme of the algorithm given
below.

\paragraph{The scheme of the recovery algorithm.}
Here we describe the steps of the algorithm and the ideas behind them. We hope
that it will help the reader to follow the detailed proofs given in the
next section.

\begin{enumerate}
    \item\label{step_procedure_wphi} The procedure of extrapolation.

        Given a vector $\gamma\in B_2^n$,
        we evaluate the function $f$ on the grid:
        $$
        y_k\approx f\left(\gamma\frac{k}{N_1}\right)=\varphi\left(\langle
        a,\gamma\rangle\frac{k}{N_1}\right),\quad
        k=-N_1,\ldots,N_1,
        $$
        where $N_1$ is large enough.
        Let $v_\gamma := n^{1/2}\langle a,\gamma \rangle$.
        Thus, $\varphi(v_\gamma t)$ will be evaluated on the uniform grid in $[-n^{-1/2},n^{-1/2}]$.
        We fit a polynomial of an appropriate degree $M$ to the  obtained values by
        least squares and use it to extrapolate $\varphi(v_\gamma t)$ to larger
        segments.
        We rely on the estimates from~\cite{DT19} to prove that the constructed
        function $\wphi_\gamma$ gives an approximation with small enough error
        $\omega_1$:
        \begin{equation}\label{phi_approx}
            |\wphi_\gamma(t)-\varphi(v_\gamma t)|\le\omega_1,
            \quad |t|\le \min(1,\frac{\sigma}{2|v_\gamma|}).
        \end{equation}

    \item \label{step_wphi} The construction of  $\wphi_i$. 

        We take  $N_2$ random vectors on the sphere:
        $\gamma^1,\ldots,\gamma^{N_2}$. For each  $\gamma^i$ we construct the function 
        $\wphi_i := \wphi_{\gamma^i}$, that approximates  $\varphi(v_it)$,
        $v_i:=v_{\gamma^i}$. As a result we obtain the set of functions
        $\{\wphi_i\}_{i=1}^ {N_2}.$

        Recall that a number $v$ is called typical if $b\le|v|\le B$. Typical $v_i$ are
        most convenient for us: we see from~\eqref{phi_approx} that, first, they
        are informative, i.e., a rather large part of $\varphi$ is recovered,
        and second, the functions $\wphi_i$ exhibit a ``good'' behaviour on a
        fairly large segment $|t|\le \sigma/(2B)$.

        Since our choice of  $\gamma^i$ is random, our algorithm is
        probabilistic. All other constructions and statements are true under
        condition~\eqref{distr2} on $\gamma^i$, that holds with
        high probability. That condition implies, e.g., that most of the
        $v_i$ are typical.

    \item \label{step_smallvar}
The estimation of function oscillation. 
        
        At this step we distinguish the case of a function $\varphi$ close to a
        constant from the case of a function whose oscillation is large.
        We estimate the oscillation of functions $\wphi_i$ and obtain either 
        \begin{equation}\label{phi_const}
            \Delta_{\sigma_1/4} := \max_{|t|\le \sigma_1/4}|\varphi(t)-\varphi(0)| \le \omega_2,
        \end{equation}
        or the inequality 
        \begin{equation}\label{delta_bound}
            \Delta_{\sigma_1} = \max_{|t|\le
            \sigma_1}|\varphi(t)-\varphi(0)|\ge\frac{\omega_2}{2}. 
        \end{equation}
        The inequality~\eqref{phi_const} for small enough $\omega_2$ leads to
        the global bound
        \begin{equation}\label{phi_const_global}
            \max_{|t|\le 1}|\varphi(t)-\varphi(0)|\le \frac{\omega_*}{2}.
        \end{equation}
        Hence $f$ can be approximated by $f(0)$ and the algorithm stops.
        In the case of~\eqref{delta_bound} we proceed to the next step.

    \item \label{step_procedure_embed}
        The procedure of the search for the embedding $\wphi_{\gamma_1}\hookrightarrow\wphi_{\gamma_2}$.

        We do not know the values of  $v_\gamma$ but a simple idea 
        helps us to approximate the ratio $|v_{\gamma_2}|/|v_{\gamma_1}|$ for
        any pair of vectors  $\gamma_1,\gamma_2$. Namely, if
        $|v_{\gamma_2}|\ge|v_{\gamma_1}|$, then from~(\ref{phi_approx})  it
        follows that $\wphi_{\gamma_1}(t)\approx \wphi_{\gamma_2}(\pm t/\lambda)$
        for some $\lambda\ge 1$. If this approximate equality holds then we
        call it an ``embedding''
        $\wphi_{\gamma_1}\hookrightarrow\wphi_{\gamma_2}$.
        We show that if $v_{\gamma_1}$ is typical, the corresponding $\lambda$ can be found
        with high accuracy and
        $|\widetilde{\lambda}(\wphi_{\gamma_1},\wphi_{\gamma_2})-\frac{|v_{\gamma_2}|}{|v_{\gamma_1}|}|\le\omega_3$.

    \item \label{step_typical} 
       The search for a typical $v$. The goal of this step is to find an index  $i_0$, such that 
        $v_{i_0}$ is typical and $|v_{i_0}|\le 3/4$.
        We find all possible pairwise embeddings
        $\wphi_i\hookrightarrow\wphi_j$  for the set of functions
        $\{\wphi_i\}$. That allows us to compare (approximately) pairs $|v_i|$,
        $|v_j|$ and analyze the order statistics of the set $\{|v_i|\}$ in order to find
        $v_{i_0}$.

    \item \label{step_vossta} The recovery of the vector  $a$.  Using the function
        $\wphi_{i_0}$ from the previous step we construct embeddings
        $\wphi_{i_0}\hookrightarrow \wphi_{\gamma}$ for appropriate vectors
        $\gamma$ (linear combinations of standard basis vectors  $e_k$), and we find
        approximate values of  $a_k/|v_{i_0}|$, where  $a_k$ are the coordinates
        of  $a$. As a result we obtain the approximation $\widetilde{a}$ to the
        vector  $a$ with error      
        \begin{equation}\label{vosst_a}
            |a-\widetilde{a}|\le C\omega_3 \le \frac{\omega_*}{2L_\sigma}.
        \end{equation}

    \item \label{step_vesotrezok}
        The recovery of  $\varphi$. The good approximation for  $a$ allows us to approximate 
        $\varphi(t)$ for any  $t\in[-1,1]$: $\varphi(t)\approx
        f(t\widetilde{a})$. We can compute $\varphi$
        on a sufficiently fine uniform grid of size $2N_3+1$ in  $[-1,1]$ and apply the technique
        of~\cite{DT19} to approximate  $\varphi$ by the polynomial of degree
        $M_1$. As a result we get
        \begin{equation}
            \label{approx_vphi}\max\limits_{|t|\le 1}|\varphi(t)-\wphi(t)|\le
            \frac{\omega_*}2.
        \end{equation}

\end{enumerate}

Finally, we have $\widetilde{f}(x) := \wphi(\langle \widetilde{a},x\rangle)$.
We estimate the error of the approximation, using~\eqref{vosst_a},~\eqref{approx_vphi} and the
Lipschitz property of $\varphi$ :
\begin{multline*}
|f(x)-\widetilde{f}(x)| \le |\varphi(\langle a,x\rangle) -
\varphi(\langle\widetilde{a},x\rangle)| +
|\varphi(\langle\widetilde{a},x\rangle)-\wphi(\langle\widetilde{a},x\rangle)| \le \\
    \le L_\sigma |a-\widetilde{a}| + \frac{\omega_*}{2} \le \omega_*.
\end{multline*}

\section{The algorithm}

In this section we describe the recovery algorithm and estimate its accuracy. The theorem follows from these considerations.  Each subsection corresponds to one step of the algorithm. 

\subsection{The procedure of extrapolation}

We will apply the following useful statement on extrapolation of analytic
functions from their values on a uniform grid. Recall that $E_\rho$ is the
ellipse with focii $\pm1$ and the sum of semi-axes $\rho$.

\begin{lemma}[See.~\cite{DT19}, Corollary 2, 4]\label{lem_DT}
    Let $\psi$ be analytic in  $E_\rho$ and $|\psi(z)|\le Q$ there; let the
    values $y_k=\psi(k/N)+\xi_k$, $k=-N,\ldots,N$ be known with
    accuracy    $|\xi_k|\le\eps$. Let  $p_M$~be the polynomial of degree not
    exceeding  $M$ that minimizes $\sum_{k=-N}^N|p(k/N)-y_k|^2$, and
    $M\le\sqrt{N/2}$. Then:
    \begin{itemize}
        \item[(i)] interpolation: for  $|x|\le 1$, we have
            \begin{equation}\label{pm_int}
                |p_M(x)-\psi(x)|\le CM^{3/2}\left(Q\frac{\rho^{-M}}{\rho-1} +
                \eps\right),
            \end{equation}
            with $C$ an absolute constant;
        \item[(ii)] extrapolation: for  $|x|\in[1,\frac12(\rho+\rho^{-1}))$, we have 
            \begin{equation}\label{pm_ext}
                |p_M(x)-\psi(x)|\le
                C\left(Q\left(\frac{M^{3/2}}{\rho-1}+\frac{r}{1-r}\right)r^M + M^{3/2}(\rho
                r)^M\eps\right),
            \end{equation}
    with $r=(|x|+\sqrt{x^2-1})/\rho$, and  $C$ an absolute constant.
    \end{itemize}
\end{lemma}

\begin{proof}
    The case  (i) corresponds to Corollary~2 from~\cite{DT19}, and the case
    (ii) to Corollary~4.    We apply
    Theorems 3 and 4 from~\cite{DT19} in order to obtain necessary estimates for singular numbers.
\end{proof}

Let us recall the setup of our algorithm. We recover an unknown function
$f(x)=\varphi(\langle a,x\rangle)$ with analytic $\varphi\in H(\sigma,1)$,
from its noisy evaluations of accuracy $\eps$.

\begin{procedure}[\texttt{extrapolation}]
Given any vector $\gamma\in B_2^n$, we receive the values $y_k\approx
    f(\gamma\frac{k}{N_1})$, $k=-N_1,\ldots,N_1$, take the polynomial $p_M$ of
    degree $\le M$ that minimizes $\sum_{k=-N_1}^{N_1}|p_M(k/N_1)-y_k|^2$, and
output the function $\wphi_\gamma(t):=p_M(n^{1/2}t)$.
\end{procedure}

\begin{proposition}
    Suppose that $N_1=2M^2$ and the inequalities hold:
    \begin{equation}\label{M_cond}
        M \ge C(\sigma)\log \omega_1^{-1},
    \end{equation}
    \begin{equation}\label{eps_cond}
        \log(1/\eps)\ge C(\log\omega_1^{-1} + M\log n).
    \end{equation}
    Then the function $\wphi_\gamma$ given by the extrapolation Procedure
    satisfies the inequality~(\ref{phi_approx}).
\end{proposition}

\begin{proof}
In our procedure we evaluate values of the function
$\psi(z)=\varphi(uz)$, $u:=\langle a,\gamma\rangle$, on the uniform grid
$\{k/N_1\}_{k=-N_1}^{N_1}$ in $[-1,1]$. One can rewrite the target
    inequality~\eqref{phi_approx} in the following way:
    \begin{equation}
        \label{psi_approx}
    |p_M(x)-\psi(x)|\le\omega_1,\quad |x|\le\min(n^{1/2},R/2),\quad
    R:=\sigma/|u|.
    \end{equation}

We start with the case $R\ge 2$.
The function  $\psi$ is analytic in the disk of radius $R$. Therefore it is
    analytic in the disk of radius  $R_1=\min(R,2n^{1/2})$ and in the set
    $E_{\rho_1}$, $\rho_1=R_1+\sqrt{R_1^2-1}$. We need 
    an estimate valid for $|x|\le R_1/2$. For such  $x$ we have 
$$
r=\frac{|x|+\sqrt{x^2-1}}{\rho_1} \le \frac{R_1/2+\sqrt{R_1^2/4-1}}{R_1+\sqrt{R_1^2-1}}
\le \frac12.
$$
    In the (more difficult) case of extrapolation we apply~(\ref{pm_ext}) for
    $\rho_1$, $Q=1$, $1\le |x|\le R_1/2$, $r\le1/2$. Since $\rho_1\ge R_1 \ge 2$ and  $\rho_1 \le 2R_1 \le
4n^{1/2}$, we get
$$
|p_M(x)-\psi(x)|\le CM^{3/2}(2^{-M}+(\rho_1/2)^M\eps) \le
CM^{3/2}(2^{-M}+(4n)^{M/2}\eps).
$$
We want each summand not to exceed  $\omega_1/2$. Therefore the  first summand
    imposes the condition on $M$: $CM^{3/2}2^{-M}\le\omega_1/2$, that holds under~\eqref{M_cond}.
    The second summand imposes the condition~(\ref{eps_cond}).
The estimate in the case  $|x|\le 1$ (interpolation) can be done similarly, the conditions on  $M$ and 
$\eps$ are weaker than in the extrapolation case.

Let us  consider the case  $R<2$. Since we have to estimate $|p_M(x)-\psi(x)|$
    for $|x|\le \min(\sqrt{n},R/2)<1$, we apply the interpolation error
    estimate~(\ref{pm_int}). Since $E_{1+\sigma}\subset\Pi_\sigma$, we see that
    $\varphi$ is analytic and bounded in $E_{1+\sigma}$; hence  as $|u|\le 1$,
    the  function $\psi$ is also analytic and
    bounded in this set. Using the inequality~(\ref{pm_int}) for
    $\rho=1+\sigma$, we get  the condition on  $M$:
    $CM^{3/2}(\sigma^{-1}(1+\sigma)^{-M}+\varepsilon) \le \omega_1/2$ that leads to~(\ref{M_cond}).
    The requirement on  $\eps$ is weaker than~(\ref{eps_cond}).

    The inequality~\eqref{psi_approx} is proven.
\end{proof}

\subsection{The construction of $\wphi_i$}

We take  $N_2$ random  vectors (uniformly) on the unit sphere:
$\gamma^1,\ldots,\gamma^{N_2} \in S^{n-1}$. For each $\gamma^i$, using the
procedure described above we construct the function $\wphi_i :=
\wphi_{\gamma^i}$, that approximates  $\varphi(v_it)$,
where $v_i:=v_{\gamma^i}=n^{1/2}\langle a,\gamma^i\rangle$. 

\paragraph{The statistics of $v_i$.}

Let us consider the sequence $v_i$.
We trivially have $|v_i|\le n^{1/2}$.
We do not know the values of  $v_i$ but we know their distribution:
$\{v_i\}_{i=1}^{N_2}$ is a sample from the random variable 
$V=n^{1/2}X_1$, where the vector  $X=(X_1,\ldots,X_n)$
is uniformly distributed on the sphere $S^{n-1}$.  Let $F^n$
be the distribution function of the r.v. $|V|$; $\widehat{F}^n_{N_2}$ be the empirical
distribution function of the sequence $|v_1|,\ldots,|v_{N_2}|$ (a
sample from the distribution $F^n$).

The distribution of $V$ is close to the  gaussian distribution:
$\mathsf{Law}(V)\approx \mathcal N(0,1)$. Let us denote by $\Phi$ and $\Phi^*$
the distribution functions of a standard gaussian variable $\xi$ and the
variable $|\xi|$, respectively.
It is known that (see, e.g.,~\cite{K06})
\begin{equation}\label{distr1}
\sup_{x\in\R}|\Phi^*(x)-F^n(x)|\le c/n.
\end{equation}
We estimate the difference between  $F^n$ and $\widehat{F}^n_{N_2}$ using the 
Dvoretzky-Kiefer-Wolfowitz inequality~\cite{M90}:
$$
\mathsf{P}(\sup_{x\in\R}|F^n(x)-\widehat{F}^n_{N_2}(x)|>\lambda) \le
2\exp(-2N_2\lambda^2).
$$

Let $\lambda=1/300$, then with probability  $1-2\exp(-N_2/45000)$ we have
\begin{equation}\label{distr2}
    \sup_{x\in\R}|F^n(x)-\widehat{F}^n_{N_2}(x)|\le \frac{1}{300}.
\end{equation}

\textbf{From here on we assume that~(\ref{distr2}) holds}.
It happens with probability not less than  $1-\delta_*$ if we take  $N_2\ge
C\ln(2/\delta_*)$.

If $n$ is sufficiently large, then $c/n\le 1/300$ in~\eqref{distr1} and
from~\eqref{distr1},~\eqref{distr2} we get for any $q_2>q_1>0$:
\begin{equation}\label{distr3}
\left|\frac{1}{N_2}\#\{i\colon q_1 \le |v_i| \le q_2\} -
    (\Phi^*(q_2)-\Phi^*(q_1))\right| \le \frac{1}{100}.
\end{equation}

In the proof we will use the following approximate values for $\Phi^*$:
$\Phi^*(2)=0.9544\ldots$, $\Phi^*(1/2)=0.3829\ldots$,
$\Phi^*(0.75)=0.5467\ldots$, $\Phi^*(0.46)=0.3544\ldots$, $\Phi^*(0.62)=0.4647\ldots$,
$\Phi^*(0.57)=0.4313\ldots$.

In particular, from the inequality $\Phi^*(2)-\Phi^*(1/2) > 0.57$, we see:
\begin{equation}\label{half}
    \frac{1}{N_2}\#\{i\colon 1/2 \le |v_i| \le 2\} > 1/2.
\end{equation}
Recall that we call a number $v$ typical, if $b\le|v|\le B$.
The appropriate choice of $b$ and $B$ gives that the number of nontypical  $v_i$ does not exceed 2 percents.
Formally, we assume that the numbers  $b,B$ are such that
\begin{equation}\label{bB_def}
\Phi^*(B)-\Phi^*(b)>\frac{99}{100},
\end{equation}
and therefore we have 
\begin{equation}
\label{bB_prop}
\#\{i\colon b\le |v_i|\le B\}\ge \frac{98}{100} N_2.
\end{equation}

\subsection{The estimation of function oscillation}

We consider the oscillation of the function $\varphi$:
$$
\Delta_h := \max_{|t|\le h} |\varphi(t)-\varphi(0)|.
$$

In order to estimate it we replace  $\varphi$ with  $\wphi_i$, and the maximum
over the segment with the maximum over the grid with some step $\nu>0$:
$$
\widetilde{\Delta}_{h,\nu}^i := \max_{k\colon |k\nu|\le
h}|\wphi_i(k\nu)-\wphi_i(0)|,\quad i=1,\ldots,N_2.
$$

    Let $\widetilde{\Delta}_{h,\nu}^\mathsf{med}$ be the median of
    the sequence $\{\widetilde{\Delta}_{h,\nu}^i\}_{i=1}^{N_2}$.
    Let us prove that if $h\le\sigma/4$ and $\nu$ is sufficiently small, say, $\nu \le
    \omega_1(2L_\sigma)^{-1}$, then the inequality holds:
    \begin{equation}
        \label{median}
    \Delta_{h/2} - 3\omega_1 \le \widetilde{\Delta}_{h,\nu}^\mathsf{med} \le
    \Delta_{2h} + 3\omega_1.
    \end{equation}

    Indeed, let $i$ be such that $1/2 \le |v_i| \le 2$. Since $h\le
    \sigma/4$, the approximation~(\ref{phi_approx}) works for $|t|\le h$ and we can replace  $\wphi_i$ with $\varphi$:
    $$
    |\widetilde{\Delta}^i_{h,\nu} - \max_{k\colon|k\nu|\le
    h}|\varphi(v_ik\nu)-\varphi(0)|| \le 2\omega_1.
    $$
    Using the Lipschitz condition and the inequality  $|v_i|\le 2$, we get
    $$
    |\Delta_{|v_i|h} - \max_{k\colon|k\nu|\le h}|\varphi(v_ik\nu)-\varphi(0)|| \le
    2L_\sigma\nu \le \omega_1
    $$
    the last inequality is true because of the restriction on $\nu$. From this $|\widetilde{\Delta}^i_{h,\nu} -
    \Delta_{|v_i|h}|\le 3\omega_1$.
    Therefore for all $i$ such that $1/2\le |v_i| \le 2$ we have
    $$
    \Delta_{h/2}-3\omega_1 \le \widetilde{\Delta}^i_{h,\nu} \le
    \Delta_{2h}+3\omega_1.
    $$
    Since it holds for  more than a half of all indices  $i$, the last
    inequalities hold also for the median.

We will use~\eqref{median} for $h:=\sigma_1/2$ (recall that
$\sigma_1=\sigma\frac{b}{4B}$) and $\nu:=\omega_1/(2L_\sigma)$.
If $\widetilde{\Delta}^{\mathsf{med}}_{h,\nu} \ge \omega_2 - 3\omega_1$, then
$$
\Delta_{2h}\ge \widetilde{\Delta}^{\mathsf{med}}_{h,\nu} -3\omega_1 \ge \omega_2 - 6\omega_1 \ge
\frac12\omega_2
$$
(we assume $\omega_2 \ge 12\omega_1$).
Therefore we obtain~\eqref{delta_bound} and proceed to the next step of the algorithm.
In the opposite case of $\widetilde{\Delta}^{\mathsf{med}}_{h,\nu} < \omega_2 -
3\omega_1$ we have $\Delta_{\sigma_1/4}\le \omega_2$ and we get~(\ref{phi_const}).
Let us derive the global bound~\eqref{phi_const_global}.

\begin{lemma}
    Let $\psi\in H(\sigma,1)$, $\sigma\in(0,1]$, and
    let $\Delta_{\tau}:=\max_{|t|\le\tau}|\psi(t)|$, $\tau\in(0,1)$. Then the estimate holds
    $$
    \max_{|t|\le 1}|\psi(t)| \le A\Delta_\tau^\alpha,\quad
    \alpha=\alpha(\sigma,\tau)\in(0,1),\;A=A(\sigma,\tau)>0.
    $$
\end{lemma}

\begin{proof}
    It is known that: for functions analytic in $E_\rho$ and bounded there by $1$
    for any $M$ there exists a polynomial $p$ of degree not exceeding $M$,
    such that $\max_{|t|\le 1}|p_M(t)-\psi(t)|\le 2\rho^{-M}/(\rho-1)$ (see \cite{Ber}).
    Since $E_{1+\sigma}\subset \Pi_\sigma$, we can take $\rho=1+\sigma$. 
    In order to make the error of the approximation less than 
    $\Delta_\tau$, we take the minimal $M$ such that
    $(1+\sigma)^{-M}\le \frac{\sigma}{2} \Delta_\tau$.
    Then $|p_M(t)-\psi(t)|\le\Delta_\tau$ and $|p_M(t)|\le 2\Delta_\tau$ for  $|t|\le\tau$.
    
   We use  Chebyshev's inequality in the following form  (see~\cite{N64},
    p.233): let $\tau_1=\tau(1+2q)$, $q>0$, and $P$ be a real algebraic
    polynomial  of degree $M$, then we have
    $$
    \frac{\|P\|_{C[-\tau_1,\tau_1]}}{\|P\|_{C[-\tau,\tau]}} \le T_M(\tau_1/\tau)
    \le(1+2q+2\sqrt{q+q^2})^M.
    $$
     $T_M$ is the classical Chebyshev polynomial of degree $M.$
    If $q\le 1$ then the last expression is less than $(1+5\sqrt{q})^M$; take
    $q=c_1\sigma^2$ to make $(1+5\sqrt{q})\le(1+\sigma)^{1/2}$ (it holds for sufficiently small $c_1$ and $\sigma<2$). Then 
    for our polynomial $p_M$ we obtain
    $$
    \max_{|t|\le \tau_1}|p_M(t)|\le 2\Delta_\tau (1+5\sqrt{q})^M = 2\Delta_\tau
    (1+\sigma)^{M/2} \le C(\sigma)\Delta_\tau^{1/2}.
    $$ 
    Hence $\Delta_{\tau_1}\le C(\sigma)\Delta_\tau^{1/2}$ for $\tau_1/\tau=1+c_1\sigma^2$.
    We iterate this construction, using points  $\tau_1,\tau_2,\ldots$, and
obtain the required estimate on $\max_{|t|\le 1}|\psi(t)|$.
\end{proof}


Applying the lemma to the function $\frac12(\varphi(t)-\varphi(0))$, we get the estimate~(\ref{phi_const_global}) under condition
\begin{equation}\label{omega2_cond}
    A(\sigma,\sigma_1/4)(\omega_2/2)^{\alpha(\sigma,\sigma_1/4)} \le \omega_*/4.
\end{equation}

So, the algorithm stops with the approximation $f\approx f(0)$ (here we require
that $\eps\le \omega_*/2$).

\subsection{The procedure of search for the embedding $\wphi_{\gamma_1}\hookrightarrow\wphi_{\gamma_2}$}

\begin{definition}
We say that function $h_1$ \textit{embeds} into $h_2$ with coefficient
$\lambda\ge 1$ and accuracy $\delta\ge 0$
(the notation is: $h_1\embed{\lambda}{\delta}h_2$), when
$$
\max_{|t|\le 1} |h_1(t)-h_2(t/\lambda)|\le \delta.
$$

We can similarly  define the embedding $h_1\embednet{\lambda}{\delta}{\nu}h_2$
replacing the maximum over $[-1,1]$ with the maximum over the grid of
step $\nu>0$:
$$
\max_{k\colon|k\nu|\le 1}|h_1(k\nu)-h_2(k\nu/\lambda)|\le \delta.
$$
\end{definition}

\paragraph{General embedding.}

Let us consider the following general situation. 

Let $g$ be some function on $[-R,R]$ and $v_1,v_2>0$ be  reals.
We suppose that we have some approximations $\widetilde{g}_i$ to the functions $g_i(t):=g(v_it)$:
\begin{equation}\label{g_i}
|\widetilde{g}_i(t)-g_i(t)|\le \omega \quad\mbox{for $|t|\le \min(1,R/v_i)$},\quad
i=1,2.
\end{equation}
We assume the functions $\widetilde{g}_i$ to be defined on $[-1,1]$.
The condition $v_i|t|\le R$ is natural since the function $g_i$ is defined for such
$t$.

For any $v_1\le R$ and any $v_2\ge v_1$ we have the embedding 
$g_1\embed{\lambda^\circ}{0}g_2$, where $\lambda^\circ =
v_2/v_1$. Indeed, for
$|t|\le 1$ we have  $|v_1t|\le R$, and therefore the functions  $g_1(t)$ and
$g_2(t/\lambda^0)$ are correctly defined and
$$
g_1(t) = g(v_1t) =g(v_2t/\lambda^0) = g_2(t/\lambda^0).
$$
The coefficient of embedding $\lambda$ is uniquely defined in the case of a
continuous function $h$. It follows from the following
simple fact: if $h$ is continuous and $h(\theta s)\equiv h(s)$ for some
$\theta\in(0,1)$, then $h(s)\equiv h(0)$.
We need the quantitative analogue of this fact when the accuracy of our
embedding is non-zero.

\begin{lemma}
    \label{lem_uniq}
    Let $h\colon[-r,r]\to\mathbb C$ be an $L$--Lipschitz function: $|h(s)-h(t)|\le L|s-t|$
    and $\Delta:=\max_{|t|\le r}|h(t)-h(0)|$. If
    $\theta\in(0,1)$ and
    $$
    \max_{|t|\le r}|h(t)-h(\theta t)|\le\delta\le\Delta/2,
    $$
    then    $$
    |1 - \theta| < \frac{\ln(2Lr/\Delta)}{\lfloor \frac{\Delta}{2\delta} \rfloor}.
    $$
\end{lemma}

\begin{proof}
    Let $\Delta=|h(\hat{t})-h(0)|$ for some $\hat t$. We see that
$$
    \Delta=|h(\hat{t})-h(0)|\le
    \sum_{j=0}^{k-1}|h(\theta^{j+1}\hat{t})-h(\theta^j\hat{t})| + 
    |h(\theta^k \hat{t})-h(0)| \le k\delta + Lr\theta^k.
$$
Setting $k=\lfloor \Delta / (2\delta) \rfloor$ we obtain
$\theta^k \ge \Delta/(2Lr)$ and
$$
    \ln(\Delta/(2Lr)) \le k\ln\theta = k\ln(1-(1-\theta))\le -k(1-\theta).
$$
\end{proof}

We can now state

\begin{lemma}
    \label{lem_embed}
    Let $0<r<R$ and $g\colon[-R,R]\to\mathbb R$ be an $L$--Lipschitz
    function, such that the condition~(\ref{g_i}) holds for some $v_1,v_2,\omega>0$; let $\lambda^\circ := v_2/v_1$.
    Suppose that $v_1\in[r,R/2]$.
    Then for any $\lambda\ge 1$ we have:
    \begin{itemize}
        \item[(i)] If $v_2\ge v_1$ and $|\lambda -
            \lambda^\circ|\le\min(\frac{\omega}{RL},\frac12)$,
            then
            $\widetilde{g}_1\embed{\lambda}{3\omega}\widetilde{g}_2$.
        \item[(ii)] If
            $\widetilde{g}_1\embednet{\lambda}{3\omega}{\nu}\widetilde{g}_2$,
            $L\nu\max(v_1,v_2)\le \omega$, and
            $$
            \Delta_r:=\max_{|t|\le r}|g(t)-g(0)| \ge 14\omega,
            $$
            then
            $$
            \frac{|\lambda-\lambda^\circ|}{\max(\lambda,\lambda^\circ)} \le
            28\omega\Delta_r^{-1}\ln(2Lr/\Delta_r).
            $$
    \end{itemize}
\end{lemma}

\begin{proof}
    (i). 
    Let $t\in[-1,1]$.
    We apply the triangle inequality:
    \begin{multline*}
    |\widetilde{g}_1(t)-\widetilde{g}_2(t/\lambda)| \le
    \underbrace{|\widetilde{g}_1(t)-g_1(t)|}_{\le\omega} +
    \underbrace{|g_1(t)-g_2(t/\lambda^\circ)|}_{=0}+\\
        +\underbrace{|g_2(t/\lambda^\circ)-g_2(t/\lambda)|}_{\le
        Lv_2|1/\lambda-1/\lambda^\circ|\le\omega}+
        \underbrace{|g_2(t/\lambda)-\widetilde{g}_2(t/\lambda)|}_{\le\omega}
        \le 3\omega.
    \end{multline*}
    The first term is bounded by $\omega$ since~\eqref{g_i} holds and
    $\min(1,R/v_1)=1$. Next,
    $g_2(t/\lambda^\circ)$ is defined and equals
    $g(v_1t)=g_1(t)$.
    The third term is estimated using Lipschitz condition:
    $$
    |g(v_2t/\lambda^\circ)-g(v_2t/\lambda)| \le Lv_2|1/\lambda-1/\lambda^\circ| =
    \frac{Lv_2}{\lambda^\circ}\cdot\lambda^{-1}|\lambda-\lambda^\circ| \le
    \frac{LR}{2}\lambda^{-1}\frac{\omega}{RL} < \omega.
    $$
    Finally, note that $g_2(t/\lambda)$ is defined and the approximation
    $g_2\approx\widetilde{g}_2$ holds: indeed,
    we have $\lambda\ge\frac12\lambda^\circ$ (as $\lambda^\circ\ge 1$ and
    $|\lambda-\lambda^\circ|\le\frac12$), so
    $$
    v_2|t|/\lambda \le 2v_2/\lambda^\circ = 2v_1 \le R.
    $$

    (ii). Consider the grid where we have the embedding:
    $\{t_k=k\nu\colon |k\nu|\le 1\}$. Pick some $t\in[-1,1]$. One can find a
    point $t_k$ in the grid such that $|t_k|\le |t|$ and $|t-t_k|\le\nu$. As in
    (i), we have a chain of  inequalities:
    \begin{multline*}
    |g_1(t)-g_2(t/\lambda)| \le \underbrace{|g_1(t)-g_1(t_k)|}_{\le
        Lv_1\nu\le\omega} +
        \underbrace{|g_1(t_k)-\widetilde{g}_1(t_k)|}_{\le\omega} + 
        \underbrace{|\widetilde{g}_1(t_k)-\widetilde{g}_2(t_k/\lambda)|}_{\le3\omega} + \\
        +\underbrace{|\widetilde{g}_2(t_k/\lambda)-g_2(t_k/\lambda)|}_{\le\omega}
        +\underbrace{|g_2(t_k/\lambda)-g_2(t/\lambda)|}_{\le
        Lv_2\nu/\lambda\le\omega} \le 7\omega.
    \end{multline*}
    The third term is estimated by the definition of the embedding 
    $\widetilde{g}_1\embednet{\lambda}{3\omega}{\nu}\widetilde{g}_2$.
    The bounds for the two last terms are valid whenever the expressions
    $g_2(t_k/\lambda)$ and $g_2(t/\lambda)$ are defined. For this, we need that
    $v_2|t|/\lambda \le R$. So, we have obtained the following inequality:
    \begin{equation}\label{selfsim}
    |g(v_1t)-g(v_2t/\lambda)|\le 7\omega\quad\mbox{for $|t|\le\min(1,R\lambda/v_2)$.}
    \end{equation}
    We consider two cases: $\lambda>\lambda^\circ$ and $\lambda<\lambda^\circ$
    (if lambdas are equal, there is nothing to prove). If
    $\lambda>\lambda^\circ$, we set $\theta=\lambda^\circ/\lambda<1$,
    $s=v_1t$, and from~(\ref{selfsim}) we obtain
    $$
    |g(s)-g(\theta s)|\le 7\omega
    $$
    when $|s|\le v_1\min(1,R\lambda/v_2)$.
    Since $R\lambda/v_2>R\lambda^\circ/v_2=R/v_1>1$, the above inequality  is
    true for $|s|\le v_1$, hence for all $|s|\le r$.

    In the second case, $\lambda<\lambda^\circ$, the inequality~(\ref{selfsim})
    can be written as
    $$
    |g(\tilde\theta\tilde s)-g(\tilde s)|\le 7\omega,\quad 
    \tilde\theta:=\lambda/\lambda^\circ,\;\tilde s:=v_2t/\lambda,
    $$
    and it holds for
    $$
    |\tilde s|\le \frac{v_2}{\lambda}\min(1,\frac{R\lambda}{v_2})=\min(v_2/\lambda,R).
    $$
    Since $v_2/\lambda>v_2/\lambda^\circ=v_1\ge r$, it holds for $|\tilde s|\le r$.

    Applying in both cases Lemma~\ref{lem_uniq}, we arrive at
    $$
    \frac{|\lambda-\lambda^\circ|}{\max(\lambda,\lambda^\circ)}\le
    \frac{\ln(2Lr/\Delta_r)}{\lfloor \Delta_r/(14\omega)\rfloor} \le
    28\omega\Delta_r^{-1}\ln(2Lr/\Delta_r).
    $$
\end{proof}

Let us apply the general constructions in our setting. The role of the function
$g$ is played by the function $\varphi(\eta t)$, where $\eta:=b\sigma_1^{-1}$; also
let $r:=b$ and $R:=2B$. Note that a typical $v_i$ satisfies $|v_i|\in[r,R/2]$, as
required in Lemma~\ref{lem_embed} and also that we have the
inequality~(\ref{delta_bound}), which bounds the oscillation of $\varphi(\eta t)$
on the segment $[-\eta\sigma_1,\eta\sigma_1]=[-b,b]$, as required. The function
$\varphi(\eta t)$ is $L_\sigma'$--Lipschitz, $L_\sigma':=\eta L_\sigma$.

\begin{procedure}[\texttt{embedding}]
For any pair of functions $\wphi_{\gamma_1}$, $\wphi_{\gamma_2}$ we try to find 
$\lambda\in[1,\lambda^{\mathrm{max}}]$, $\lambda^{\mathrm{max}}:=n^{1/2}b^{-1}$, such that 
$$
\wphi_{\gamma_1}(\eta t)\embednet{\lambda}{3\omega_1}{\nu}
\wphi_{\gamma_2}(\eta t),
$$
    where $\nu=n^{-1/2}(L_\sigma')^{-1}\omega_1$, we search for the parameter $\lambda$
    over the grid in $[1,\lambda^{\mathrm{max}}]$ of  step
    $\omega_1/(BL_\sigma')$. If we obtain an embedding with some $\lambda$, we set
    $\widetilde{\lambda}(\wphi_{\gamma_1},\wphi_{\gamma_2}):=\lambda$ and write
$\wphi_{\gamma_1}\hookrightarrow\wphi_{\gamma_2}$. Otherwise, we similarly try
    to embed $\wphi_{\gamma_1}(-t)$ into $\wphi_{\gamma_2}$ and in the case of
    success we set
    $\widetilde{\lambda}(\wphi_{\gamma_1},\wphi_{\gamma_2}):=\lambda$ and also
    write $\wphi_{\gamma_1}\hookrightarrow\wphi_{\gamma_2}$. If we
    fail to find such $\lambda$, we say that there is no embedding and denote this situation
    as $\wphi_{\gamma_1}\not\hookrightarrow\wphi_{\gamma_2}$.
\end{procedure}

\begin{proposition}
    \label{prop_embed}
    Let $v_{\gamma_1}$ be a typical number, the values $\omega_i$ satisfy the
    inequalities
    \begin{equation}\label{omegas}
        \omega_2\ge28\omega_1,\quad C\omega_1\omega_2^{-1}\ln(4L_\sigma'b/\omega_2) \le
        \omega_3 bn^{-1/2}.
    \end{equation}
    If $|v_{\gamma_2}|\ge |v_{\gamma_1}|$, then 
    $\wphi_{\gamma_1}\hookrightarrow\wphi_{\gamma_2}$. If
    $\wphi_{\gamma_1}\hookrightarrow\wphi_{\gamma_2}$, then 
    $|v_{\gamma_2}|\ge|v_{\gamma_1}|(1-\omega_3)$.
 In both cases 
    $|\widetilde{\lambda}(\wphi_{\gamma_1},\wphi_{\gamma_2})-\frac{|v_{\gamma_2}|}{|v_{\gamma_1}|}|\le\omega_3$.
\end{proposition}

\begin{proof}
    We apply Lemma~\ref{lem_embed} with $g(t)=\varphi(\eta t)$, $r=b$, $R=2B$,
    $\omega=\omega_1$, $L=L_\sigma'$,
    $\widetilde{g}_1(t)=\wphi_{\gamma_1}(\eta t)$,
    $\widetilde{g}_2(t)=\wphi_{\gamma_2}(\eta t)$.
    Note that the typical $v_{\gamma_1}$ satisfies $|v_{\gamma_1}|\in[r,R/2]$.
    Let $v_{\gamma_1}>0$, $v_{\gamma_2}>0$; other cases are similar.
    Recall the notation $\lambda^\circ = v_{\gamma_2}/v_{\gamma_1}$.

    Let $v_{\gamma_2}\ge v_{\gamma_1}$. We have 
    $\lambda^\circ \le n^{1/2}/b=\lambda^{\mathrm{max}}$;
    therefore, we try to embed with some $\lambda$,
    $|\lambda-\lambda^\circ|\le \omega_1/(2BL)$
    (it is clear that $\omega_1/(2BL)<1/2$). By Lemma~\ref{lem_embed}, we get
    the embedding with accuracy $\le 3\omega_1$ on the whole segment, hence also on the grid;
    thus, $\wphi_{\gamma_1}\hookrightarrow\wphi_{\gamma_2}$.

    Now suppose that $\wphi_{\gamma_1}\hookrightarrow\wphi_{\gamma_2}$. Recall the
    inequality~(\ref{delta_bound}) obtained at the step 2. It gives 
    the estimate for $g$ on the segment $[-b,b]$ since $\eta \sigma_1=b$:
    $$
    \Delta = \max_{|t|\le b}|g(t)-g(0)|\ge\frac12\omega_2.
    $$
    We may apply (ii) from Lemma~\ref{lem_embed}; note that the condition
    $\omega_2/2\ge 14\omega_1$ is satisfied by~\eqref{omegas}. We obtain 
    \begin{equation}
    \label{lambda_approx}
    \frac{|\lambda-\lambda^\circ|}{\max(\lambda,\lambda^\circ)} \le
    C\omega_1\omega_2^{-1}\ln(4Lb/\omega_2) \le \omega_3
    /\lambda^{\mathrm{max}}.
    \end{equation}
    Due to the construction, we have $\max(\lambda,\lambda^\circ)\le \lambda^{\mathrm{max}}$, and consequently 
    $|\lambda-\lambda^\circ|\le \omega_3$. Thus, 
    $$
    v_{\gamma_2}=v_{\gamma_1}\lambda^\circ = v_{\gamma_1}(\lambda - (\lambda-\lambda^\circ)) \ge
    v_{\gamma_1}(1-\omega_3).
    $$
\end{proof}

\subsection{The search for typical $v_i$.}

For all pairs $i,j\in\{1,\ldots,N_2\}$ we start the embedding procedure to search for
possible embeddings $\wphi_i\hookrightarrow\wphi_j$.
It will allow us to compare the values of different $|v_i|$ with the relative error 
$\omega_3$. In the following we use 
the estimates~(\ref{distr3}),~(\ref{bB_prop}) and Proposition~\ref{prop_embed}.

Let us calculate the numbers 
$$
K_j := \#\{i\colon \wphi_i\hookrightarrow \wphi_j\}.
$$
Denote by $J_0$ the set of $j$ such that $0.4 \le K_j/N_2 \le 0.5$.
Let us consider the set $\{v_j, j\in J_0\}$.

Let us show that if $j\in J_0$, then $|v_j|\le 0.75$.
Suppose the converse, let $|v_j|>0.75$. Use that $\Phi^*(0.75) > 0.54$; hence 
due to~(\ref{distr3}), there are at least $0.53N_2$ indices $i$, such that $|v_i|\le
0.75$; there are at least $0.51N_2$ typical numbers among them.
Proposition~\ref{prop_embed} implies that all 
such $\wphi_i$ embed into $\wphi_j$, this contradicts $j\in J_0$.

Now we show that $j\in J_0$ implies $|v_j|\ge 0.45$. Let $|v_j|<0.45$ and
$\wphi_i\hookrightarrow\wphi_j$. If $v_i$ is typical, then due to 
Proposition~\ref{prop_embed} we have $|v_i|\le 0.46$ (we may assume that
$\omega_3$ is small). The amount of such $v_i$ is at most
 $N_2(\Phi^*(0.46)+0.01)\le 0.37N_2$. The amount of non-typical $v_i$ is at most 
 $0.02N_2$. Therefore, there are at most $0.39N_2$ such $i$, this contradicts 
$j\in J_0$.

Arguing as above, we see that indices $j$ with $|v_j|\in[0.57,0.62]$ belong to 
$J_0$. Therefore, the set $J_0$ is non-empty:
$$
\#J_0 \ge N_2(\Phi^*(0.62)-\Phi^*(0.57) - 0.01) \ge 0.02N_2.
$$
We take an arbitrary $i_0\in J_0$ and obtain $|v_{i_0}|\in
[0.45,0.75]$. Moreover, $v_{i_0}$ is typical.

\subsection{The recovery of the vector $a$}

Here we use the extrapolation and embedding procedures and also the
function $\wphi_{i_0}$ with typical $|v_{i_0}|\le 3/4$ constructed at the
previous step.

For every $k=1,\ldots,n$ we try to embed $\wphi_{i_0}\hookrightarrow
\wphi_{e_k}$ and find the corresponding $\widetilde{\lambda}$ (using the
\texttt{embedding} procedure). Since $a$ is a unit vector, 
for at least one $k$ we have $n^{1/2}|a_k|\ge 1 > |v_{i_0}|$ and thus the embedding exists. 
Suppose that for the $k=k^*$ the corresponding $\widetilde{\lambda}$ is maximal.
We have $|n^{1/2}|a_{k^*}|/|v_{i_0}| -
\widetilde{\lambda}(\wphi_{i_0},\wphi_{e_{k^*}})|\le\omega_3$. Although 
$|a_{k^*}|$ is not necessarily maximal, in any case  
$\max|a_k|\le 1.01|a_{k^*}|$.

Without loss of generality we may assume that $a_{k^*}>0$.
Indeed, as we already noticed in Introduction,
the function $f$ is invariant under the substitution $(a,\varphi(x))\mapsto (-a,\varphi(-x))$.

Thus, we know the ratio $n^{1/2}a_{k^*}/|v_{i_0}|$ up to $\omega_3$;
let us determine similar ratios for all other $k$. For a given $k$ consider the vector $\gamma =
0.9e_{k^*} + 0.1e_k \in B_2^n$. Then
$$
v_\gamma = n^{1/2}\langle 0.9e_{k^*} + 0.1e_k, a\rangle = n^{1/2}(0.9a_{k^*} +
0.1 a_k) > 0.75 \ge |v_{i_0}|,
$$
therefore there exists an embedding $\wphi_{i_0}\hookrightarrow\wphi_\gamma$ and  
$|\widetilde{\lambda}(\wphi_{i_0},\wphi_\gamma)-v_\gamma/|v_{i_0}||\le \omega_3$ holds. From
the equality $n^{1/2}a_k=-9n^{1/2}a_{k^*} +10v_\gamma$ and from the obtained
inequalities on $\widetilde{\lambda}(\wphi_{i_0},\wphi_\gamma)
,\widetilde{\lambda}(\wphi_{i_0},\wphi_{e_{k^*}})$ it follows that $$
\left|\frac{n^{1/2}a_k}{|v_{i_0}|}-10\widetilde{\lambda}(\wphi_{i_0},\wphi_\gamma) + 
9\widetilde{\lambda}(\wphi_{i_0},\wphi_{e_{k^*}})\right| \le C\omega_3.
$$

Thus, we approximate coordinatewise the vector $n^{1/2}a/|v_{i_0}|$ by a vector
$w$ (the coordinates of $w$ are determined from the previous inequality)
with the accuracy $O(\omega_3)$. So, we have $|w-n^{1/2}a/|v_{i_0}||\le
Cn^{1/2}\omega_3$.
It is easy to see that 
$$
|w-ra|\le\delta \quad\Rightarrow\quad
\left|\frac{w}{|w|}-a\right|\le\frac{2\delta}{r-\delta},
$$
which implies that for $\widetilde{a}:=w/|w|$ we have $|\widetilde{a}-a|\le C\omega_3$.
We obtain~(\ref{vosst_a}). At this step we have used $2n(2N_1+1)$ additional
evaluations (we needed the functions $\wphi_{e_k}$ and
$\wphi_{0.9e_{k^*}+0.1e_k}$, $1\le k\le n$).

\subsection{The recovery of $\varphi$}

The knowledge of $\widetilde{a}$ allows us to approximate $\varphi$ at an arbitrary point of the segment 
$[-1,1]$. Namely, for $t\in[-1,1]$ we evaluate $f$ at the point $t\widetilde{a}$ and obtain the approximation $y_t$:
$|y_t-f(t\widetilde{a})|=|y_t-\varphi(t\langle a,\widetilde{a}\rangle)|\le \eps$. Hence 
$$
|y_t-\varphi(t)|\le|\varphi(t\langle a,\widetilde{a}\rangle) - \varphi(t)|+\eps
\le L_\sigma|a-\widetilde{a}| + \eps \le CL_\sigma \omega_3 + \eps;
$$
where we applied the inequality $|\langle a,\widetilde{a}\rangle -1|\le |a-\widetilde{a}|.$
We use this method to obtain approximate values of $\varphi$  on
some grid $\{k/N_3\colon k=-N_3,\ldots,N_3\}$. Then we construct the
polynomial $p_{M_1}$ of degree $M_1$ using the least squares fit, as in the
Step 1. Lemma~\ref{lem_DT} (i) gives us an estimate for the accuracy 
$$
|p_{M_1}(x)-\varphi(x)|\le
C_1M_1^{3/2}(\frac{\rho^{-M_1}}{\rho-1}+L_\sigma\omega_3 + \eps),
\quad -1 \le x \le 1.
$$
We can take $\rho=1+\sigma$.
For our purposes it is sufficient  that each of the summands is at most $\omega_*/6$.
The first summand requires that $M_1\ge C(\sigma)\log(6/\omega_*)$. The condition for the second summand is 
\begin{equation}\label{omega3_last}
    \omega_3 \le \frac{\omega_*}{CL_\sigma}M_1^{-3/2}.
\end{equation}
The condition on $\eps$ is
\begin{equation}\label{eps_cond2}
    \log(1/\eps) \ge \log(6/\omega_*) + C\log M_1.
\end{equation}

So we can put $N_3=2M_1^2$. At this step we used $2N_3+1$ values of $f$.

\subsection{About the choice of parameters}

We choose parameters in the following order:
\begin{itemize}
    \item $N_2=\lceil C\ln(2/\delta_*)\rceil$.
    \item $M_1=\lceil C(\sigma)\log(2/\omega_*)\rceil$, $N_3=2M_1^2$.
    \item We choose $\omega_3$ based on the condition~(\ref{omega3_last}). It automatically implies  
    the condition on $\omega_3$ from~(\ref{vosst_a}).
    \item We choose $\omega_2$ based on condition~(\ref{omega2_cond}).
    \item We choose $\omega_1$ based on inequality~(\ref{omegas}) and
        condition $\omega_2\ge 28\omega_1$.
    \item We define $M$ by the condition~(\ref{M_cond}) and suppose $N_1 = 2M^2$.
    \item The main requirement on $\eps$ is~(\ref{eps_cond}) and (in fact,
        weaker) condition~\eqref{eps_cond2}.
\end{itemize}

Total number of function evaluations equals $N=(2N_1+1)(N_2+2n)+2N_3+1$.

\section{The implementation of the recovery algorithm} 
\label{sect_practice}

We implemented the recovery algorithm  with some modifications
using Python language and the \texttt{numpy} library. The code is
available on Github:~\cite{gitTZ}.
Let us describe the  main changes that we introduced,
and the results of numerical experiments.


The most important modification is related to the embedding procedure. The optimal parameter
$\lambda$ for embedding a polynomial $p_1(t)$ into polynomial $p_2(t)$ can be
found as the solution of the problem $\min \limits_{-1 \le \mu \le 1} \max
\limits_{|t| \le 1} |p_1(t) - p_2(\mu t)|$,
where $\mu = \frac{1}{\lambda}$. Instead of the brute force search  for possible
$\mu$ on the grid and the estimation of $C$-norm, we consider the $L_2$-norm and
explicitly find 
\begin{equation}
    \label{effective_embed}
\min \limits_{-1 \le \mu \le 1} S(\mu),\quad \mbox{where }S(\mu) := \int
\limits_{|t| \le 1} (p_1(t) - p_2(\mu t))^2\,dt.
\end{equation}
Indeed, $S(\mu)$ is a polynomial itself and it can be minimized on the set
$\{-1,1\}\cup\{\mu\in(-1,1)\colon S'(\mu)=0\}$. This approach allows us to
effectively calculate rather accurate estimates of the embedding coefficient. 

To test our algorithm we recover functions of the form 
\begin{equation}
\label{phi_test}
\varphi(x) := x^{K_1}\left(\frac{A_0}{\sqrt{2}} + \sum_{k=1}^{K_2} A_k \cos(\pi
kx) + B_k \sin(\pi kx)\right).
\end{equation}


The values of parameters used in the  program are given in the table below: 

$$
\begin{array}{c|c|c|c|c|c|c|c|c|c}
    n  & \varepsilon & M & M_1 & N_1 & N_2 & N_3 & K_1 & K_2 \\ \hline
    50 & 10^{-20}   & 12 & 30  & 200 & 25  & 200 & 8   & 7 \\ 
\end{array}
$$

We used the values of $\wphi_\gamma(t)$ only for
$t\in[-n^{-1/2},n^{-1/2}]$, i.e. only the interpolation was done (extrapolation
was not required);
the oscillation in our experiments was large enough to search for embeddings
on the segment $[-n^{-1/2},n^{-1/2}]$ and find lambdas with high
accuracy.

We note that the error $\varepsilon$ is very small in our implementation since
the values of the function $\varphi$ near $0$ are also very small. It
is possible to increase the parameter $\varepsilon$ but in this case we should
decrease the degree $K_1$. Also, we can consider a relative error instead of an
absolute error, in this case we can take it equal to e.g. $10^{-5}$.

The following table contains the results of $10$ random experiments made with 
the parameters given above. In each experiment  
the vector of the coefficients of trigonometric polynomial defining $\varphi$ by~\eqref{phi_test}
and the vector $a$ of the ridge function were
chosen randomly and uniformly from the unit spheres.
We also specify the value of the $\alpha$ parameter~\eqref{alpha}.

$$\begin{array}{c|c|c|c|c|c}

\mbox{\#} & |\widetilde a - a|_{2} & \|\wphi - \varphi\|_{C} & \alpha \\ \hline

1 & 1.2691\cdot10^{-5} & 4.7683\cdot10^{-10} & 5.7276\cdot10^{-7} \\ 
2 & 4.5421\cdot10^{-5} & 6.4352\cdot10^{-9} & 3.4025\cdot10^{-6}\\
3 & 1.6969\cdot10^{-4} & 6.3703\cdot10^{-8} & 1.6590\cdot10^{-6}\\
4 & 9.2308\cdot10^{-5} & 1.2412\cdot10^{-8} & 1.5862\cdot10^{-6} \\
5 & 2.2896\cdot10^{-7} & 3.0620\cdot10^{-13} & 5.0647\cdot10^{-6} \\
6 & 6.1517\cdot10^{-6} & 2.3048\cdot10^{-10} & 9.2345\cdot10^{-6} \\
7 & 4.2528\cdot10^{-6} & 6.8383\cdot10^{-11} & 1.9411\cdot10^{-6} \\
8 & 5.4939\cdot10^{-7} & 1.3600\cdot10^{-12} & 2.5017\cdot10^{-6} \\
9 & 2.5037\cdot10^{-6} & 4.3252\cdot10^{-11} & 3.1520\cdot10^{-7}  \\
10 & 4.8359\cdot10^{-5} & 2.6839\cdot10^{-9} & 1.4242\cdot10^{-6} \\

\end{array}$$

The example of the function $\varphi$ of the form ~\eqref{phi_test} is shown in the Fig.~\ref{pic1}.

\begin{figure}[ht]
\centering
\includegraphics[width = 0.8\linewidth]{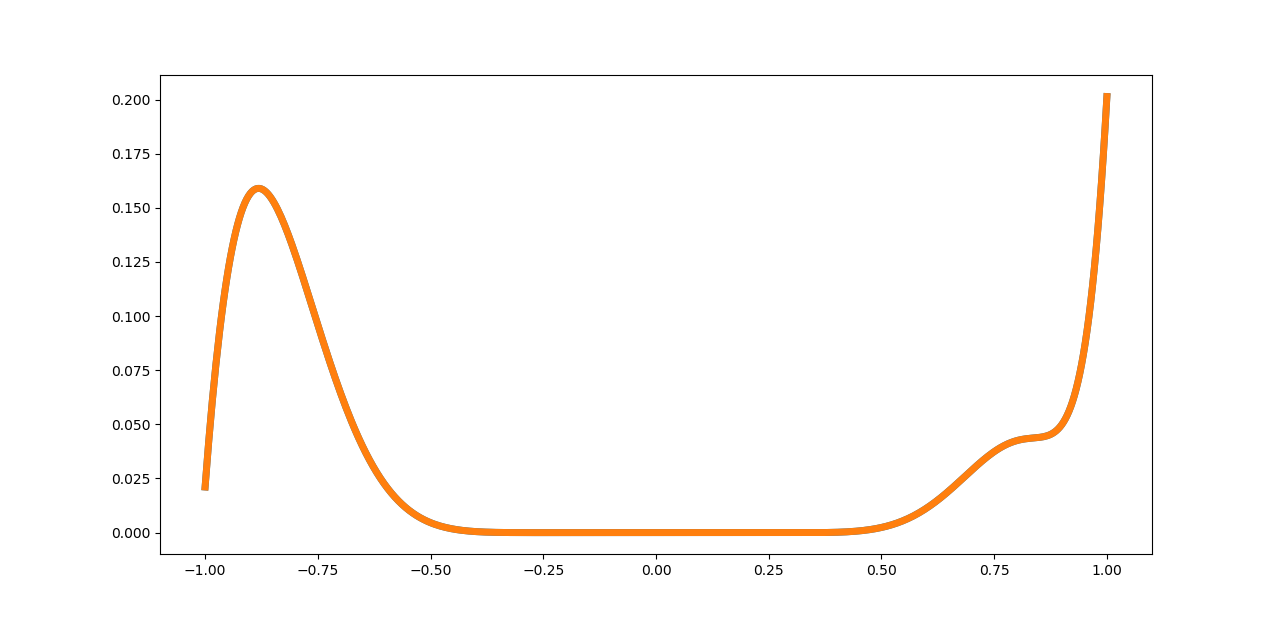}
\caption{The function $\varphi$ (example)}
\label{pic1}
\end{figure}

\section{Further work}

\paragraph{Non-analytic functions.} 
Although the classes $\mathcal R(H(\sigma,Q),B_2^n)$ that we have considered are
rich enough, e.g., they include polynomial ridge functions $f(x)=p(\langle
a,x\rangle)$, it seems that the assumption that $\varphi$ is analytic may be
weakened.

The extrapolation technique of~\cite{DT19} may work for functions $\varphi$ with
certain restrictions on the order of the decay of sequences $\|\varphi^{(k)}\|_{C[-1,1]}$ (as
$k\to\infty$) or $E_k(f):=\inf\limits_{\deg p\le k}\|\varphi-p\|_{C[-1,1]}$.

We also note that the embedding technique may work without extrapolation in some
cases; a lower bound on the oscillation $\max\limits_{|t|\le
h}|\varphi(t)-\varphi(0)|$,
$h\approx n^{-1/2}$ is crucial here. Nevertheless we have to deal
with the case of small oscillation (that was the goal of
Step~\ref{step_smallvar}).

\paragraph{Derandomization.}
Our algorithm is probabilistic, so the natural question is: can we get rid of
the randomness? The only point where we need it is
the condition~\eqref{distr2}. 
Let $(X,\mu)$ be a probability space and $\mathcal C$ is some family of
measurable subsets $C\subset X$. Recall that the discrepancy of a finite set $\Gamma$ 
for a family $\mathcal C$ is defined as
$$
\mathrm{disc}(\Gamma,\mathcal C) := \sup_{C\in\mathcal
C}\left|\mu(C)-\frac{|\Gamma\cap C|}{|\Gamma|}\right|.
$$
The condition~\eqref{distr2} is equivalent to the following
discrepancy bound:
\begin{equation}
    \label{disc}
    \mathrm{disc}(\{\gamma_i\}_{i=1}^{N_2},\mathcal C_n) \le \frac{1}{300},
\end{equation}
where $\mathcal C_n$ is the set of symmetric spherical caps $\{\gamma\in
S^{n-1}\colon |\langle a,\gamma\rangle|>t\}$.
There are two difficulties: first, we need a bound for the discrepancy as
$n\to\infty$; but the discrepancy theory is mostly developed for fixed $n$.
Second, we want a deterministic construction of $\{\gamma_i\}$. In the case of
boxes $\mathcal R_n$ in $[0,1]^n$ it is known~\cite[Prop.~6.72]{Tbook} that
there are constructive sets $\Xi_N$ of $N$ points with
$$
\mathrm{disc}(\Xi_N,\mathcal R_n)\le Cn^{3/2} N^{-1/2} \ln^{1/2}\max(n,N).
$$
So, $N=n^{3+o(1)}$ points would suffice for small discrepancy. It is interesting
to get good constructive bounds in the spherical case.

\paragraph{General ridge functions.}
A ridge function is a simple yet interesting object, but in practice we have
to deal with more complex functions. So it is important to study the 
recovery of generalized ridge functions, e.g., sums $\varphi_1(\langle
a_1,x\rangle)+\varphi_2(\langle a_2,x\rangle)+\ldots+\varphi_r(\langle
a_r,x\rangle)$. One may consider the case of analytic (or even polynomial) functions $\varphi_j$.

\paragraph{Acknowledgement.}
The authors wish to express their gratitude to S.V.~Konyagin for his
suggestion to consider the case of regular ridge functions and constant
encouragement and to the anonymous referees for their careful work and precise
comments.


\begin{thebibliography}{XXXXXX}

\bibitem{Ber}
Bernstein, S.: Sur la meilleure approximation des fonctions
continues par les polynomes du degr\'e donn\'e. II. 
Communications de la Soci\'et\'e math\'ematique de Kharkow, 
        2-\'ee s\'erie \textbf{13}, 145–194 (1912) (in Russian) 

\bibitem{BP}
Buhmann, M. D., Pinkus A.: Identifying linear combinations of ridge functions. 
Adv. Appl. Math. \textbf{22}, 103–118 (1999) 


\bibitem{CDD12}  
Cohen, A., Daubechies, I., DeVore, R., Kerkyacharian, G., Picard, D.: 
Capturing ridge functions in high dimensions from point queries. 
Constr. Approx. \textbf{35}, 225--243 (2012) 

\bibitem{CW}
Casazza, P.G., Woodland, L.M.: 
Phase retrieval by vectors and projections. Operator Methods in Wavelets, Tilings, and Frames. 
Contemp. Math. \textbf{626}, 1--7 (2014)

\bibitem{DT19}
Demanet, L., Townsend, A.: 
Stable extrapolation of analytic functions. 
Found. Comput. Math. \textbf{19}, 297--331 (2019) 

\bibitem{DM21} 
Doerr, B., Mayer, S.: 
The recovery of ridge functions on the hypercube suffers from the curse of dimensionality. 
J. Complexity \textbf{63} (2021) 

\bibitem{FSV12}
Fornasier, M., Schnass, K., Vybiral, J.: 
Learning functions of few arbitrary linear parameters in high dimensions. 
Found. Comput. Math. \textbf{12}, 229--262 (2012) 

\bibitem{GL07} 
Gaiffas, S., Lecue, G.: 
Optimal rates and adaptation in the single-index model using aggregation. 
Electron. J. Stat. \textbf{1}, 538--573 (2007)

\bibitem{DKPW}
Dick, J., Kritzer, P., Pillichshammer, F., Wo\ '{z}niakowski, H.: 
Approximation of analytic functions in Korobov spaces. 
J. Complexity \textbf{30}, 2–28 (2014)

\bibitem{K06}
Khokhlov, V.I.: 
The uniform distribution on a sphere in $\mathbb{R}^s$. Properties of projections. I.
Theory Probab. Appl. \textbf{50}, 386--399 (2006) 

\bibitem{KKM}
Konyagin, S. V., Kuleshov, A. A., Maiorov, V. E.: 
Some Problems in the Theory of Ridge Functions. 
Proc. Steklov Inst. Math. \textbf{301}, 144–169 (2018) 

\bibitem{M90}
Massart, P.: 
The tight constant in the Dvoretzky–Kiefer–Wolfowitz inequality. 
Ann. Probab. \textbf{18}, 1269--1283 (1990) 

\bibitem{MUV15} 
Mayer, S., Ullrich, T., Vybiral, J.: 
Entropy and sampling numbers of classes of ridge functions. 
Constr. Approx. \textbf{42}, 231--264 (2015) 

\bibitem{N64} 
Natanson, I.P.:
Constructive function theory, vol. 1, Ungar (1964) 

\bibitem{NW16}
Novak, E., Wozniakowski, H.: 
Tractability of multivariate problems for standard and linear information in the worst case setting, Part I. 
J. Approx. Theory \textbf{207}, 177--192 (2016) 


\bibitem{Os00}
Osipenko, K.Yu: 
Optimal Recovery of Analytic Functions.  
Nova Publishers (2000) 
 
  
\bibitem{P99}
Pinkus, A.: 
Approximation theory of the MLP model in neural networks.  
Acta Numer. \textbf{8}, 143–195 (1999) 
 
\bibitem{Pbook} 
Pinkus, A.: 
Ridge Functions. 
Cambridge Tracts in Math. 205, Cambridge University Press, Cambridge (2015) 

\bibitem{Sima} 
Š\'ima, J.: 
Training a single sigmoidal neuron is hard. 
Neural Comput. \textbf{14}, 2709--2728 (2002) 

 
\bibitem{Tbook}
Temlyakov, V.: 
Greedy Approximation. 
Cambridge University Press (2011) 

\bibitem{TC14} 
Tyagi, H., Cevher, V.: 
Learning non-parametric basis independent models from point queries via low-rank methods. 
Appl. Comput. Harmon. Anal. \textbf{37}, 389--412 (2014) 

\bibitem{Vyb14}
Vyb\'{ı}ral, J.: 
Weak and quasi-polynomial tractability of approximation of infinitely differentiable functions. 
J. Complexity \textbf{30}, 48--55 (2014) 


\bibitem{ZMR20}
Zaitseva, T.I., Malykhin, Yu.V., Ryutin, K.S.: 
On Recovery of Regular Ridge Functions. 
Math. Notes \textbf{109}, 307--311 (2021) 

\bibitem{gitTZ}
    \texttt{https://github.com/TZZZZ/new\_ridge\_no\_curse}
 
\end{thebibliography}
\end{document}